\documentclass[11pt]{amsart}

\usepackage{amsmath,amssymb,amsthm}
\usepackage{graphicx}
\usepackage{subfigure}
\usepackage{enumerate}
\usepackage{fullpage}
\usepackage{url}
\usepackage{float}
\usepackage{xspace}
\usepackage[hidelinks]{hyperref}
\usepackage{color}
\usepackage{thmtools, thm-restate}
\usepackage[numbers]{natbib}

\newtheorem{theorem}{Theorem}
\newtheorem{conj}{Conjecture}
\newtheorem{cor}{Corollary}
\newtheorem{lem}{Lemma}[section]
\newtheorem{prop}{Proposition}

\newtheorem{remark}{Remark}

\title{New constructions of optimal arrangements of \\ $2d$ lines in $\mathbb{C}^d$}

\author{Alexey Glazyrin}

\address{Alexey Glazyrin, School of Mathematical \& Statistical Sciences, The University of Texas Rio Grande Valley, Brownsville, TX 78520, USA}
\email{alexey.glazyrin@utrgv.edu}

\date{}

\begin{document}

\maketitle
\begin{abstract}
In this paper we provide new constructions of equiangular tight frames of size $2d$ in $\mathbb{C}^d$. We generalize the doubling construction of Fallon and Iverson to a tensor multiplication construction based on a suitable pair consisting of a complex Hadamard matrix and an equiangular tight frame. In particular, such a pair always exists whenever there is an amicable pair of real Hadamard matrices. Most notably, amicable Hadamard pairs of order $q+1$ exist for all prime powers $q\equiv 3\pmod 4$. We also find specific constructions based on a family of pairs of order 6 and on pairs whose equiangular tight frames are defined by Paley conference matrices with $q\equiv 1\pmod 4$. Finally, we provide a power construction of equiangular tight frames that generalizes the construction of Turyn for conference matrices.
\end{abstract}

\section{Introduction}

An optimal packing of $N$ points in a compact metric space is a configuration of $N$ points maximizing the minimal distance between them \cite{CHS96}. In this paper we study optimal packings in the complex projective space $\mathbb{CP}^{d-1}$. Slightly reformulating the problem, we would like to find a set of unit complex vectors $\{x_1,\ldots,x_N\}\subset\mathbb{C}^d$ such that the maximum absolute inner product $\max\limits_{i\neq j}|\langle x_i, x_j \rangle|$, or \textit{coherence of the configuration}, is as small as possible. The universal bound for this problem is due to Welch \cite{W74} who proved that for any set of $N$ unit complex vectors in $\mathbb{C}^d$ there is a pair $(x_i, x_j)$ such that
\begin{equation}\label{eq:welch}
|\langle x_i, x_j \rangle|^2 \geq \frac {N-d}{d(N-1)}. 
\end{equation}
Configurations of points attaining this bound are known as \textit{equiangular tight frames} (ETFs) of size $d\times N$.

A general problem is to find pairs of $N$ and $d$ for which $d\times N$ ETFs exist. In the real case, it is equivalent to the existence of certain strongly regular graphs \cite{W09}. For the complex case, which is the main subject of this paper, it is known that $d\leq N \leq d^2$ \cite{DGS75, SH03, STDH07}; see also a recent improvement \cite{FJM26} which excludes $(d^2-d+1, d^2)$ from the range of potential sizes of ETFs. The existence of ETFs for $N=d^2$ is a major open problem due to Zauner \cite{Z11}. For a survey on existence of equiangular tight frames with given parameters, see \cite{FM15}. For a general overview of tight frames we refer to the comprehensive text of Waldron \cite{W18}.

A convenient way of describing an ETF is via its \textit{signature matrix}. Let $G$ be an $N\times N$ Gram matrix of a set of unit vectors representing a $d\times N$ ETF. Let $\alpha=\sqrt{\frac {N-d}{d(N-1)}}$ be the coherence of this configuration as determined by the Welch bound (\ref{eq:welch}). Then the signature matrix $S$ of this ETF is determined by $G=I+\alpha S$. The following proposition of Holmes and Paulsen \cite{HP04} provides the description of ETFs via their signature matrices.

\begin{prop}\cite{HP04}\label{prop:signature}
$S$ is a signature matrix of a $d\times N$ equiangular tight frame if and only if the following conditions are satisfied: $S=S^*$, $S$ has zero diagonal entries and unimodular off-diagonal entries, and $S^2=cS+(N-1)I$, where
$$c=(N-2d) \sqrt{\frac {N-1}{d(N-d)}}.$$
\end{prop}

If a matrix $S$ satisfies Proposition \ref{prop:signature}, we will call it a $(c,N)$-matrix or $c$-matrix as it is sometimes more convenient to describe $S$ by the parameter $c$ than by the dimension $d$ of the corresponding frame. 

The case $N=2d$ is especially interesting to us since the condition in Proposition \ref{prop:signature} reads $S^2=(N-1)I$ and connects ETFs of this type to Hadamard and conference matrices (see Section \ref{sect:hadamard}). Recently, Fallon and Iverson studied $d\times 2d$ ETFs and developed the doubling construction that produces frames without direct derivation from Hadamard or conference matrices \cite{FI25}. They conjectured that $d\times 2d$ ETFs exist in all dimensions.

\begin{conj}\cite{FI25}\label{conj:double}
For every positive integer $d\geq 2$, there exists a $d\times 2d$ ETF in $\mathbb{CP}^{d-1}$.
\end{conj}

This conjecture has been verified for all $d\leq 165$ \cite{IJM24}. Moreover, explicit ETFs are known for all $d\leq 150$ except in dimensions $77, 93, 105, 133$. Group frames of size $d\times 2d$ have also been studied in \cite {BCGO25}.

In this paper we present new constructions of $d\times 2d$ ETFs that generalize, respectively, the doubling procedure of Fallon and Iverson and the power construction of Turyn for conference matrices \cite{T71}. The paper is organized as follows. In Section \ref{sect:hadamard} we provide necessary definitions and basic constructions of $d\times 2d$ ETFs via Hadamard and conference matrices. In Section \ref{sect:kronecker} we develop the new construction of ETFs via Kronecker multiplication and explain its consequences for the $d\times 2d$ case. In Section \ref{sect:power} we show how the power construction of Turyn for conference matrices extends to the case of ETFs.

\section{Hadamard and conference matrices}\label{sect:hadamard}

In this section, we provide general definitions, necessary background on Hadamard and conference matrices, and constructions of equiangular tight frames based on Hadamard and conference matrices.

\textit{A Hadamard matrix} is a square matrix $H$ of order $N$ with all entries $\pm 1$ satisfying $HH^{\text{T}}=NI$. Such matrices first appeared in a paper of Hadamard \cite{H93}, who showed that they may exist only for $N=1, 2$ or multiples of 4. Their existence for all such $N$ is a famous open problem \cite{P33}. \textit{A complex Hadamard matrix} is a square matrix $H$ of order $N$ with unimodular complex entries satisfying $H H^* = NI$. In contrast to the real case, complex Hadamard matrices exist for all $N$ \cite{S67, TZ06}. 

A Hadamard matrix $H$ is symmetric if $H=H^{\text{T}}$. A Hadamard matrix $H$ is called skew Hadamard if $H=C+I$, where $C$ is a skew-symmetric matrix. Both symmetric and skew Hadamard matrices of order $N$ are widely believed to exist for all $N$ divisible by 4 \cite{GS70, SY92} (see \cite{CP24} for the current status of all these conjectures). A pair of Hadamard matrices $(X, Y)$ is called amicable if $X$ is symmetric, $Y$ is skew Hadamard, and $XY^{\text{T}} = YX^{\text{T}}$. The notion of amicable pairs appeared in \cite{W70,W71} and, conjecturally, amicable pairs also exist for all $N$ divisible by 4 \cite{SY92}.

\textit{A conference matrix} is a square matrix $C$ of order $N$ with zero diagonal entries and $\pm 1$ off-diagonal entries satisfying $CC^{\text{T}}=(N-1)I$. Conference matrices first appeared in the work of Belevitch on telephone conference networks \cite{B68}. We note that $C$ is a skew-symmetric conference matrix if and only if $C+I$ is a skew Hadamard matrix.

There are two constructions of $d\times 2d$ ETFs based on conference matrices \cite{FI25}. If $C$ is a symmetric conference matrix of order $N$, then $C^2=(N-1)I$ and $C$ satisfies the criteria for a signature matrix from Proposition \ref{prop:signature}. Since $c=0$, $N=2d$ and $C$ defines a $d\times 2d$ ETF. If $C$ is a skew-symmetric conference matrix, then $S=iC$ satisfies the criteria from Proposition \ref{prop:signature} and, similarly to the previous case, $N=2d$ and $S$ defines a $d\times 2d$ ETF.

These two constructions can be combined with constructions of conference matrices to establish the existence of $d\times 2d$ ETFs for various dimensions. Due to the standard construction of Paley \cite{P33}, symmetric conference matrices of order $N$ exist for all $N=q+1$, where $q$ is a prime power congruent to $1\pmod 4$, and skew-symmetric conference matrices exist for all $N=q+1$, where $q$ is a prime power congruent to $3\pmod 4$, thus immediately confirming the conjecture for $d=(q+1)/2$.

Recently, Fallon and Iverson showed that there exists an $N\times 2N$ ETF if there exists a $d\times N$ ETF with $|c|\leq 1$. If $c=0$, this essentially means that one can double the size of an ETF satisfying Conjecture \ref{conj:double}. As for the case $c\neq 0$, Fallon and Iverson used the construction of Strohmer who showed the existence of an $(N/2-1)\times (N-1)$ ETF \cite{S08,R07} with $|c|<1$ whenever there is a skew Hadamard matrix of order $N$. Combining this with the Paley construction and doubling procedure confirms Conjecture \ref{conj:double} when $d$ is a prime power congruent to $3\pmod 4$.

Another doubling construction was found by Iverson, Jasper, and Mixon \cite{IJM24} who showed that the existence of a symmetric conference matrix of order $N$ implies the existence of an $(N-1)\times (2N-2)$ ETF. This, in particular, confirms Conjecture \ref{conj:double} when $d$ is a prime power congruent to $1\pmod 4$.

If $x$ is a vector in an ETF, we can replace it with $\alpha x$, where $\alpha$ is a unimodular complex number, without violating ETF conditions. This observation lies behind the notion of \textit{Seidel switching}. Two equiangular tight frames are called \textit{switching equivalent} if their vectors can be obtained via permutations and a finite number of substitutions $x\rightarrow\alpha x$, where $|\alpha|=1$. This equivalence was defined by Seidel in the real case \cite{LS66, S68}, that is, for $\alpha=-1$. Now it is used for complex ETFs as well \cite{BPT09, CW16}.

We can fix a vector $x_1$ in an ETF and switch all other vectors so that $\langle x_1, x_i \rangle$ is a positive real number for all $i$. In this case, the signature matrix $S$ of the ETF will have ones as off-diagonal entries in the first row and column. Following the literature \cite{BPT09}, we say that $S$ is \textit{a standard form} of the ETF. We also say that the matrix $W$ obtained from $S$ by deleting the first row and the first column is \textit{a core} of the ETF \cite{T71}.

The construction in Section \ref{sect:kronecker} uses manipulations of the signature matrix of an ETF to get a new signature matrix in a higher dimension. In contrast, the construction in Section \ref{sect:power} uses manipulations of the core matrix of an ETF.

\section{Kronecker multiplication construction} \label{sect:kronecker}

We call a pair $(X,Y)$ of $N\times N$ complex matrices \textit{ETF-amicable} if $X$ is a complex Hermitian Hadamard matrix, $Y$ is a signature $c$-matrix, and 
\begin{equation}\label{eq:amic}
cX=k N I_N + XY+YX
\end{equation}
for some real constant $k$. We use ETF-amicable pairs to obtain a multiplying construction for ETFs. By $\otimes$ below we mean the Kronecker product of matrices.

\begin{theorem}\label{thm:mult}
Let $(X,Y)$ be an ETF-amicable pair of $N\times N$ matrices satisfying condition (\ref{eq:amic}) above and let $S$ be a signature $(k, M)$-matrix of arbitrary order $M$. Then $B=S\otimes X + I_M \otimes Y$ is a signature $(c, MN)$-matrix.
\end{theorem}

\begin{proof}
$B$ is Hermitian because all matrices $S, X, I_M, Y$ are Hermitian. Clearly, $B$ has zero diagonal entries and unimodular off-diagonal entries. Using the definition of Hadamard matrix and Proposition \ref{prop:signature} for signature matrices we have the following three conditions.

\begin{align*}
X^2=NI_N\\
Y^2=cY+(N-1)I_N\\
S^2=kS+(M-1)I_M
\end{align*}

 It remains to check the quadratic condition on $B$.

$$B^2=(S\otimes X + I_M \otimes Y)^2=S^2\otimes X^2 + S\otimes (XY+YX) + I_M \otimes Y^2=$$

$$ (kS+(M-1)I_M) \otimes NI_N + S\otimes (XY+YX) + I_M\otimes (cY+(N-1)I_N)=$$

$$(MN-1)I_M\otimes I_N + S\otimes (kNI_N+XY+YX)+c I_M\otimes Y =$$

$$(MN-1)I_{MN} + S\otimes cX + c I_M\otimes Y = (MN-1)I_{MN} + cB.$$
In the last line we use condition (\ref{eq:amic}) from the definition of ETF-amicable pairs. By Proposition \ref{prop:signature}, $B$ is a signature $c$-matrix.
\end{proof}

\subsection{ETF-amicable pairs of order $2$}

Here we show that for ETF-amicable pairs of order 2, the construction from Theorem \ref{thm:mult} is equivalent to the doubling construction of Fallon and Iverson. Indeed, a complex Hermitian Hadamard matrix $X$ of order 2 is necessarily $\begin{pmatrix} \pm 1&\alpha\\ \overline{\alpha}&\mp 1\end{pmatrix}$, where $|\alpha|=1$. A signature matrix of order 2 is necessarily $\begin{pmatrix} 0&\beta\\ \overline{\beta}&0\end{pmatrix}$, where $|\beta|=1$ and $c$ must be 0. Checking the amicability condition, we see that any matrices of this kind work for $k=-\Re (\alpha\overline{\beta})$ which may take any value from $[-1,1]$.

\begin{cor}\cite[Theorem 6]{FI25}
If there is a complex ETF of size $d\times N$ with a signature $k$-matrix, $|k|\leq 1$, then there is a complex ETF of size $N\times 2N$.
\end{cor}

\begin{remark}
It may seem that the construction from Theorem \ref{thm:mult} is richer than the one in \cite{FI25} because it can use infinitely many pairs $(\alpha,\beta)$ for a given $k$. However, constructions for different pairs with the same $\alpha\overline{\beta}$ are just Seidel equivalent to each other.
\end{remark}

\subsection{ETF-amicable pairs and amicable pairs of Hadamard matrices}

Since we are particularly interested in the case of $d\times 2d$ ETFs, for the following result we take $c=k=0$ in Theorem \ref{thm:mult}.

\begin{cor}\label{cor:mult}
Let $X$ be a complex Hermitian Hadamard matrix of order $2N$ and $Y$ be a signature matrix of an $N\times 2N$ ETF.  Also assume $XY+YX=0$. For any $d$, if there exists a $d\times 2d$ ETF, then there exists a $2dN \times 4dN$ ETF.
\end{cor}

\begin{proof}
This is an immediate consequence of Theorem \ref{thm:mult} for $c=k=0$ and the ETF-amicable pair $(X,Y)$.
\end{proof}

The next corollary provides a particular class of ETF-amicable pairs for which Corollary \ref{cor:mult} works by connecting it to amicable pairs of Hadamard matrices defined in Section \ref{sect:hadamard}.

\begin{theorem}\label{thm:amic}
If there exists an amicable pair of $2N\times 2N$ Hadamard matrices and a $d\times 2d$ ETF, then there exists a $2dN\times 4dN$ ETF.  
\end{theorem}

\begin{proof}
Assume $(A,B)$ is an amicable pair of Hadamard matrices. By the definition of amicable pairs, $A^{\text{T}}=A$, $B=I+C$, where $C^{\text{T}}=-C$, and $AC^{\text{T}}$ is symmetric. We take $X=A$ and $Y=iC$. $X$ is a symmetric real matrix so it is Hermitian. $Y$ is a signature matrix of an $N\times 2N$ ETF as explained in Section \ref{sect:hadamard}. It remains to check the condition $XY+YX=0$.

$$XY+YX=i(AC+CA)=i(-AC^{\text{T}} + (A^{\text{T}}C^{\text{T}})^{\text{T}})=$$

$$i(-AC^{\text{T}} + (AC^{\text{T}})^{\text{T}}) = i(-AC^{\text{T}} + AC^{\text{T}}) = 0.$$
\end{proof}

The smallest nontrivial example has order 2: $A=\begin{pmatrix} 1&1\\1&-1\end{pmatrix}$, $B=\begin{pmatrix} 1&1\\-1&1\end{pmatrix}$. Using this pair we recover the doubling construction. There are, however, many other examples of amicable pairs including infinite series of various types, and all of them can be used for the multiplying construction of ETFs. In the following corollary we use the list of known amicable pairs from \cite{S13}.

\begin{cor}\label{cor:amic_concrete}
Let $N$ be one of the following:

\begin{itemize}

\item $2^\ell$;

\item $q+1$, where $q$ is a prime power, $q\equiv 3\pmod 4$;

\item $(4t-1)^{2r+1}+1$, when circulant Hadamard cores of order $4t-1$ exist;

\item A product of any of the numbers above.

\end{itemize}

Assume there exists a $d\times 2d$ ETF. Then there exists a $dN\times 2dN$ ETF.
\end{cor}

\subsection{ETF-amicable family for $N=6$}\label{subsect:6}

The following family of ETF-amicable pairs $(X(k), Y(k))$ is based on a one-parameter family of ETFs of size 6 from \cite{ET16}, which provides the matrices $Y(k)$ in the family. The corresponding matrices $X(k)$ were found computationally with AI assistance.

Let $0\leq k\leq 4/3$, and define
\[
u=-\frac{3k}{4}-i\sqrt{1-\frac{9k^2}{16}},\qquad h=1-\frac{3k}{2},\qquad \delta=\sqrt{2-h^2},
\]
\[
w=\frac{\delta-h}{2}+i\frac{\delta+h}{2}.
\]
Define
\[
X(k)=
\begin{pmatrix}
-1&1&-u&-u&-1&i\\
1&-1&-u&-u&-i&1\\
-\bar u&-\bar u&1&-i&-\bar u&\bar u\\
-\bar u&-\bar u&i&-1&\bar u&-\bar u\\
-1&i&-u&u&1&1\\
-i&1&u&-u&1&1
\end{pmatrix},
\]
\[
Y(k)=
\begin{pmatrix}
0&-1&-1&-1&w&w\\
-1&0&-1&-1&-w&-w\\
-1&-1&0&1&-1&1\\
-1&-1&1&0&1&-1\\
\bar w&-\bar w&-1&1&0&-1\\
\bar w&-\bar w&1&-1&-1&0
\end{pmatrix}.
\]

\begin{cor}\label{cor:6}
Let $0\leq k\leq 4/3$ and let $S$ be a signature $(k, M)$-matrix. Then there exists a $d\times 2d$ ETF for $d=3M$.
\end{cor}

\begin{proof}
Both $X(k)$ and $Y(k)$ are Hermitian and satisfy unimodularity conditions because $|u|=|w|=1$. Explicitly checking $X(k)^2=6I$, $Y(k)^2=5I$, $X(k)Y(k)+Y(k)X(k)=-6kI$, we see that $(X(k),Y(k))$ is an ETF-amicable pair for $c=0$ and arbitrary $k\in [0,4/3]$.
\end{proof}

The following result is analogous to the construction of Fallon and Iverson based on ETFs found by Strohmer \cite{S08}.

\begin{theorem}\label{thm:3}
If there exists a skew Hadamard matrix of order $N$, then there exists a $3(N-1)\times 6(N-1)$ ETF.
\end{theorem}

\begin{proof}
The construction of Strohmer provides an ETF of size $(N/2-1)\times (N-1)$ with signature $2/\sqrt{N}$-matrix. Therefore, Corollary \ref{cor:6} is applicable and a $d\times 2d$ ETF for $d=3(N-1)$ exists.
\end{proof}

Combining this result with the Paley construction, we verify Conjecture \ref{conj:double} for all $d=3q$, where $q$ is a prime power congruent to $3\pmod 4$. This immediately leads to an explicit construction for $d=93$ that was missing in \cite{IJM24}. The other missing explicit case in \cite{IJM24}, $d=105$, can be obtained from a skew Hadamard matrix of order 36 found in \cite{GS70}.

\subsection{ETF-amicable pairs of order $q+1$ for prime powers congruent to $1\pmod 4$}\label{subsect:1mod4}

The construction in this subsection was first found numerically for $N=10$ by AI and then generalized to all $N=q+1$, where $q$ is a prime power congruent to $1\pmod 4$.

Let $q$ be a prime power such that $q\equiv 1\pmod 4$ and $N=q+1$. For the construction we need a concrete description of the Paley construction. It is convenient to explain it in terms of the projective line $\mathbb P^1(\mathbb F_q)=\{\infty\}\cup\mathbb F_q$. Let $\chi$ be the quadratic character in $\mathbb F_q$, that is, $\chi(0)=0$, $\chi(b^2)=1$ for $b\neq 0$, $\chi(a)=-1$ if $a$ is a nonsquare. We index all rows and columns of a matrix $Y$ by the elements of $\mathbb P^1(\mathbb F_q)$ and define $Y_{\infty,\infty}=0$, $Y_{x,\infty}=Y_{\infty,x}=1$ for $x\neq\infty$, $Y_{x,y}=\chi(x-y)$ for $x,y\neq\infty$. Since $-1$ is a square for $q\equiv 1\pmod 4$, the matrix $Y$ is symmetric. It is also straightforward to check that $Y^2=(N-1)I$. The following lemma explains the construction of a matrix $X$ forming an ETF-amicable pair with $Y$.

\begin{lem}\label{lem:1mod4}
Let $q\equiv1\pmod 4$ be a prime power, and let $Y$ be the symmetric Paley conference matrix of order $N=q+1$. Then there is a Hermitian complex Hadamard matrix $X$ of order $N$ such that $XY+YX=0$.
\end{lem}

\begin{proof}
Choose an arbitrary nonsquare $a\in\mathbb F_q$, and define the fixed-point-free projective involution

\[
\tau(\infty)=0,\qquad \tau(0)=\infty,\qquad
\tau(x)=\frac{a}{x}\quad(x\neq 0, \infty).
\]

Define signs

\[
\varepsilon_\infty=1,\qquad \varepsilon_0=-1,\qquad
\varepsilon_x=\chi(x)\quad(x\neq 0, \infty).
\]

Let $Q$ be the signed permutation matrix whose rows and columns are indexed by $\mathbb P^1(\mathbb F_q)$ in the same order as for $Y$ and whose entries are determined by $Q_{x,\tau(x)}=\varepsilon_x$. Then

\[
Q^{\text{T}}=-Q,\qquad Q^2=-I,\qquad QY=-YQ.
\]
The first two identities follow immediately from $\varepsilon_x \varepsilon_{\tau(x)}=-1$. The last identity can be checked explicitly. For $x,y \neq 0,\infty$, $[QY]_{x,y}=\chi(a-xy)=-[YQ]_{x,y}$. The remaining cases are also straightforward. We note that, in particular, $[QY]_{x,\tau(x)}=0$ and all other entries of $QY$ are $\pm 1$.

Now we define $X=iQ-YQ$ and check that all conditions for an ETF-amicable pair are satisfied. From the observation above, $X$ has unimodular entries and
\[
X^* = -iQ^{\text{T}}-YQ = iQ-YQ = X,
\]

\[
X^2=-Q^2-iQYQ-iYQ^2+YQYQ = I + iYQ^2-iYQ^2 - Y^2Q^2 =NI,
\]
Finally,
\[
XY+YX = iQY-YQY+iYQ-Y^2Q = i(QY+YQ) - Y(QY+YQ) = 0.
\]
\end{proof}

\begin{remark}
Under certain conditions we can construct an ETF-amicable pair for any signature 0-matrix $Y$. In particular, if there exists a Hermitian monomial involution $M$ such that $MYM=-Y$, then we can use $X=(I+iY)M$. In Lemma \ref{lem:1mod4}, $iQ$ plays the role of $M$. A similar construction of $M$ works for prime powers $q\equiv3\pmod 4$. It is quite possible that the same approach works for all frames with some group structure such as harmonic frames \cite{FS20}.
\end{remark}

Combining Lemma \ref{lem:1mod4} with the multiplying construction of Theorem \ref{thm:mult}, we get the following result.

\begin{theorem}\label{thm:1mod4}
Let $N=q+1$, where $q$ is a prime power and $q\equiv 1\pmod 4$. If there exists a $d\times 2d$ ETF, then there exists a $dN\times 2dN$ ETF.
\end{theorem}

\section{Kronecker power construction}\label{sect:power}

The construction in this section is based on power constructions for conference matrices. For a skew-symmetric conference matrix of order $m+1$, the general problem is to find a skew-symmetric conference matrix of order $m^r+1$ for $r>1$. Goldberg \cite{G66} proved that a skew-symmetric conference matrix of order $m^3+1$ exists whenever a skew-symmetric conference matrix of order $m+1$ exists. Goethals and Seidel \cite{GS67} mentioned the earlier construction of Belevitch for $r=2$. Wallis \cite{W72} constructed similar matrices for $r=5$ and 7. Finally, Turyn \cite{T71} generalized these constructions to every $r$. The same machinery based on powers of cores is also used in the paper \cite{S13} already mentioned above.

In this section we show that the construction of Turyn works for signature matrices as well. The main idea is to modify a core of the signature matrix using Kronecker multiplication. In order to do this we need a corresponding version of the Holmes–Paulsen criterion for core matrices.

\begin{prop}\label{prop:core}
$W$ is a core of a complex $d\times 2d$ ETF if and only if $W$ is Hermitian, $W$ has zero diagonal and unimodular off-diagonal entries, $WJ_{2d-1}=0$, and $W^2+J_{2d-1}=(2d-1)I_{2d-1}$.
\end{prop}

\begin{proof}
These conditions immediately follow from Proposition \ref{prop:signature} when $S$ is in the standard form. Conversely, if $W$ satisfies these conditions, then $S=\begin{pmatrix}
0&\mathbf{1}^{\text T}\\
\mathbf{1}&W
\end{pmatrix}$
satisfies Proposition \ref{prop:signature}.
\end{proof}

The power construction for $r=2$ is different from the case $r>2$, so we show it in a separate lemma. For simplicity we use $I$ for $I_m$, $J$ for $J_m$ and include indices for $I$ and $J$ of larger orders.

\begin{lem}\label{lem:power2}
If $W$ is a core of a $\frac {m+1} 2\times (m+1)$ ETF, then $W_2=W\otimes W + I\otimes J - J\otimes I$ is a core of a $\frac {m^2+1} 2\times (m^2+1)$ ETF.
\end{lem} 

\begin{proof}
$W_2$ is Hermitian because the matrices $W$, $J$, $I$ are Hermitian. Clearly, diagonal entries of $W_2$ are 0 and off-diagonal entries are unimodular. $W_2 (J\otimes J) = m J\otimes J - m J\otimes J = 0$. Finally,

$$W_2^2=W^2\otimes W^2 + I\otimes J^2 + J^2\otimes I - 2 J\otimes J=$$

$$(mI-J)\otimes (mI-J) + m I\otimes J + m J\otimes I-2 J_{m^2}= m^2 I_{m^2} - J_{m^2}.$$
Thus, all conditions of Proposition \ref{prop:core} are satisfied so $W_2$ is a core of an ETF.
\end{proof}

Now we can prove the general result. The construction is essentially identical to the construction of Turyn.

\begin{theorem}\label{thm:power}
If there exists a $\frac {m+1} 2\times (m+1)$ ETF, then there exists a $\frac {m^r+1} 2\times (m^r+1)$ ETF for any $r\geq 2$.
\end{theorem}

\begin{proof}
It is sufficient to prove the theorem for odd $r$: if $r=2^a r'$, $a\geq 1$ and $r'$ is odd, we can use Lemma \ref{lem:power2} $a$ times and reduce the problem to the odd case. Let $W$ be a core of a $\frac {m+1} 2\times (m+1)$ ETF. In this proof, \textit{a string} means a sequence of letters $W$, $I$, and $J$ such that after $I$ there is always a $J$ and before $J$ there is always an $I$. The order is defined cyclically so this condition must be satisfied for the first and last letters of a string as well. For each string $S=(S_1, \ldots, S_r)$ we define a matrix
\[
G_S=S_1\otimes \ldots \otimes S_r.
\]
We claim that $W_r=\sum_S G_S$ is a core of an ETF.

First we show that for any 0/1-sequence $L$ of length $r$, with the exception of a sequence of all ones, there exists exactly one string $S$ such that 0 in $L$ corresponds to $W$ or $J$ in $S$ and 1 in $L$ corresponds to $I$ or $J$ in $S$. Consider blocks of consecutive ones in $L$. The first element of each such block must correspond to $I$ in $S$. The next element corresponds to $J$. The third one in the block, if it exists, corresponds to $I$ again, etc. Therefore, we can uniquely reconstruct correspondences for all ones in $L$ which immediately determines all $I$ and $J$ letters in a string. All remaining zeros in $L$ necessarily correspond to $W$ in $S$.

Using this we can deduce that for any two strings, there is $I$ in one string and $W$ in the other string at the same position. Assume this is not the case for two strings $S_1$ and $S_2$. We construct a 0/1-sequence using the following procedure: for a position in a sequence, we choose 0 if at least one of $S_1$ and $S_2$ has $W$ at this position or they both have $J$; we choose 1 if at least one of $S_1$ and $S_2$ has $I$ at this position. The constructed sequence corresponds to both $S_1$ and $S_2$, contradicting the uniqueness of a string proved above.

Note that the way strings are defined is essentially symmetric for $I$ and $J$: if we swap all $I$ and $J$ and read a string in the opposite order, we get a string again. Therefore, for any two strings, there is $J$ in one string and $W$ in the other string at the same position.

Now we are ready to show that $W_r$ possesses the properties of a core. All tensor factors are Hermitian so $W_r$ is Hermitian too. Clearly, it has zeros on the diagonal, since every string contains at least one $W$ for an odd $r$. For every off-diagonal entry, we can encode it by its position in the Kronecker product as a 0/1 sequence: if it is on the diagonal for a particular tensor factor, we choose 1; if it is not on the diagonal, we choose 0. There is a unique string $S$ corresponding to this 0/1-sequence and the matrix $G_S$ has a unimodular entry at this position so the entry of $W_r$ is unimodular.

Since $J_{m^r}=J\otimes\ldots\otimes J$ and each $S$ has at least one $W$, $G_S J_{m^r} = 0$ for every $S$. The only remaining property is $W_r^2=m^r I_{m^r} - J_{m^r}$. In order to show this we notice that $G_{S_1} G_{S_2}=0$ for different $S_1$ and $S_2$ because different strings have tensor factors $W$ and $J$ in the same position. Therefore, $W_r^2=\sum_S G_S^2$. Each $G_S^2$ contains only $W^2=mI-J$, $I^2=I$, and $J^2=mJ$ as its factors so $W_r^2$ is a linear combination of Kronecker products of $I$ and $J$.

The only way to get $I_{m^r}=I\otimes\ldots\otimes I$ or $J_{m^r}=J\otimes\ldots\otimes J$ in this combination is from the string with all $W$ letters. Their coefficients in the expansion are $m^r$ and $(-1)^r=-1$, respectively. It remains to show that the coefficients of all other Kronecker products of $I$ and $J$ vanish. Without loss of generality, take a product $I\otimes J \otimes \ldots$. It can occur in the expansion of $W_r^2$ from a string that starts with $I, J$ or from a string that starts with $W, W$. We note that $S_1=(W,W,S')$ and $S_2=(I,J,S')$ for the same sequence $S'$ of length $r-2$ are strings or non-strings simultaneously. $G_{S'}$ is defined similarly to $G_S$ as a Kronecker product of the matrices corresponding to letters in $S'$.

$$G_{S_1}^2+G_{S_2}^2=(W^2\otimes W^2 + I^2\otimes J^2)\otimes G_{S'}^2=$$

$$\left((mI-J)\otimes (mI-J) + m I\otimes J\right)\otimes G_{S'}^2= (-m+m) I\otimes J \otimes G_{S'}^2 + \ldots,$$
so the contributions of $S_1$ and $S_2$ to any term $I\otimes J\otimes \ldots$ cancel each other, which completes the proof of the theorem.

\end{proof}

\section{Discussion}\label{sect:discuss}

The paper is devoted to constructions of ETFs with $c=0$. Naturally, we can ask whether the multiplying construction of Theorem \ref{thm:mult} works for other $c$. One straightforward example is the case $c=\pm 2, k=\pm 2$, where we can take an ETF-amicable pair $(X,Y)$ such that $Y^2=cY+(N-1)I$ and $X=\pm(Y-\frac c 2 I)$. In this case, $X$ is a Hermitian Hadamard matrix with constant diagonal, a class well known in the literature \cite{Sz13}. The construction of Theorem \ref{thm:mult} just produces a Kronecker product of such complex Hadamard matrices. Apart from this concrete construction and the case $c=0$, there are no other ETF-amicable pairs. The following short spectral argument was provided by AI when answering the prompt asking about examples of pairs for other $c$.

\begin{prop}\label{prop:otherc}
ETF-amicable pairs exist only for $c=0$ or $c=\pm 2$. In the case $c=\pm 2$, the only pairs are those determined by Hermitian Hadamard matrices with constant diagonal as described above.
\end{prop} 

\begin{proof}
Let $(X,Y)$ be an ETF-amicable pair with constants $c$ and $k$. Since $Y^2=cY+(N-1)I$, $Y$ has two eigenvalues $\lambda_{1,2}=\frac {c\pm\sqrt{c^2+4(N-1)}} {2}$. We consider $X$ and $Y$ in the eigenbasis of $Y$ so that
\[Y=
\begin{pmatrix}
\lambda_1 I_p&0\\
0&\lambda_2 I_q
\end{pmatrix},\qquad
X=
\begin{pmatrix}
A&B\\
B^*&C
\end{pmatrix},
\]
where $p$ and $q$ are the respective dimensions of the eigenspaces. From the amicability condition $XY+YX+kNI=cX$, we get $A=hI_p$ and $C=-hI_q$, where $h=\frac{kN}{c-2\lambda_1}$.

The Hadamard condition $X^2=NI$ implies $BB^*=(N-h^2)I_p$ and $B^*B=(N-h^2)I_q$. We note that $BB^*$ and $B^*B$ have equal traces so either $p=q$ or $N-h^2=0$. In the former case, using that the trace of $Y$ is 0 we get $c=\lambda_1+\lambda_2=0$. In the latter case, both $BB^*=0$ and $B^*B=0$ so
\[
X=
\begin{pmatrix}
hI_p&0\\
0&-hI_q
\end{pmatrix}
=\frac {2h} {\lambda_1-\lambda_2} \left(Y-\frac c 2 I\right).
\]
Returning to the original basis, $X$ has unimodular off-diagonal entries only if $\frac {2h} {\lambda_1-\lambda_2}=\pm 1$. Then the diagonal entries of $X$ are unimodular only if $c=\pm 2$. Combining these constraints, we get $X=\pm(Y-\frac c 2 I)$.
\end{proof}

For $N=2$ and 6 there are one-parameter families of ETF-amicable pairs. It would be interesting to find ETF-amicable families for other $N$ as well. The construction for $N=6$ was built on a family of complex conference matrices of order 6 from \cite{ET16}. The paper \cite{ET16} also contains families of complex conference matrices of orders 10 and 14 based on parametrizations of complex Hadamard matrices in \cite{Sz08}. However, these families do not seem to lead to ETF-amicable pairs.  Interestingly, the multiplying by 6 construction from Subsection \ref{subsect:6} in the case $k=0$ does not seem to be equivalent to the general construction from Subsection \ref{subsect:1mod4} for $q=5$.

The medium version of Conjecture \ref{conj:double} \cite[Conjecture 20]{IJM24} states the existence of 2-circulant ETFs of size $d\times 2d$. It would be interesting to check whether the multiplying and power constructions can produce 2-circulant ETFs.

\section{Acknowledgments}

Alexey Glazyrin was partially supported by the NSF grants DMS-2054536, DMS-2349063 and the Simons Foundation's Travel Support for Mathematicians program.

\section{AI statement}

ChatGPT 4.6 Sol was used for literature search and light copyediting throughout the paper. The constructions in Subsections \ref{subsect:6}, \ref{subsect:1mod4} and the argument for Proposition \ref{prop:otherc} were found with AI assistance. All text in the paper was written by the author, who takes full responsibility for its correctness.

\bibliography{ETFs}
\bibliographystyle{plainnat}

\end{document}